\documentclass[10pt]{amsart}
\usepackage{amssymb}

\newtheorem{thm}{Theorem}[section]
\newtheorem{lemma}[thm]{Lemma}
\newtheorem{proposition}[thm]{Proposition}

\newtheorem{definition}[thm]{Definition}
\newtheorem{corollary}[thm]{Corollary}

\newcommand{\A}{\mathbb{A}}

\newcommand{\p}{\mathbb{P}}
\newcommand{\q}{\mathbb{Q}}

\newcommand{\cf}{\mathrm{cf}}
\newcommand{\cof}{\mathrm{cof}}
\newcommand{\dom}{\mathrm{dom}}

\newcommand{\add}{\textrm{Add}}

\newcommand{\col}{\mathrm{Col}}
\newcommand{\restrict}{\upharpoonright}

\begin{document}

\title{A Note on the Eightfold Way}

\author{Thomas Gilton and John Krueger}

\address{Thomas Gilton \\ Department of Mathematics \\
	University of California, Los Angeles\\
	Box 951555\\
	Los Angeles, CA 90095-1555}
\email{tdgilton@math.ucla.edu}

\address{John Krueger \\ Department of Mathematics \\ 
	University of North Texas \\
	1155 Union Circle \#311430 \\
	Denton, TX 76203}
\email{jkrueger@unt.edu}

\date{December 2018; revised June 2019}



\thanks{The second author was partially supported by 
	the National Science Foundation Grant
	No. DMS-1464859.}

\begin{abstract}
	Assuming the existence of a Mahlo cardinal, we construct a model in which there 
	exists an $\omega_2$-Aronszajn tree, the $\omega_1$-approachability property fails, 
	and every stationary subset of $\omega_2 \cap \cof(\omega)$ reflects. 
	This solves an open problem of \cite{eight}.
\end{abstract}

\maketitle

Cummings, Friedman, Magidor, Rinot, and Sinapova \cite{eight} proved the consistency 
of any logical Boolean combination of the statements which assert the $\omega_1$-approachability property, 
the tree property on $\omega_2$, and stationary reflection at $\omega_2$. 
For most of these combinations, they assumed the existence of a 
weakly compact cardinal in order to construct the desired model. 
This is a natural assumption to make, since the $\omega_2$-tree property implies that 
$\omega_2$ is weakly compact in $L$. 
On the other hand, Harrington and Shelah \cite{HS} proved that stationary reflection at 
$\omega_2$ is equiconsistent with the existence of a Mahlo cardinal. 
Cummings et al.\ \cite{eight} asked whether a Mahlo cardinal is sufficient to prove the 
consistency of the existence of an $\omega_2$-Aronszajn tree, the failure of 
the $\omega_1$-approachability property, and stationary reflection at $\omega_2$. 
In this article we answer this question in the affirmative.

We begin by reviewing the relevant definitions and facts. 
We refer the reader to \cite{eight} for a more detailed discussion of these ideas and their history. 
A stationary set $S \subseteq \omega_2 \cap \cof(\omega)$ is said to \emph{reflect} at an 
ordinal $\beta \in \omega_2 \cap \cof(\omega_1)$ if $S \cap \beta$ is a stationary subset of $\beta$. 
If $S$ does not reflect at any such ordinal, $S$ is \emph{non-reflecting}. 
We say that \emph{stationary reflection} holds at $\omega_2$ if every stationary subset of 
$\omega_2 \cap \cof(\omega)$ reflects to some ordinal in $\omega_2 \cap \cof(\omega_1)$. 

An \emph{$\omega_2$-Aronszajn tree} is a tree of height $\omega_2$, whose levels have size 
less than $\omega_2$, and which has no cofinal branches. 
The \emph{$\omega_2$-tree property} is the statement that there does not exist an $\omega_2$-Aronszajn tree. 
A well-known fact is that if the $\omega_2$-tree property holds, then $\omega_2$ is a weakly compact 
cardinal in $L$. 
Therefore, if one starts with a Mahlo cardinal $\kappa$ which is not weakly compact in $L$ 
(for example, if $\kappa$ is the least Mahlo cardinal in $L$), then in any subsequent forcing extension in 
which $\kappa$ equals $\omega_2$, there exists an $\omega_2$-Aronszajn tree.

The \emph{$\omega_1$-approachability property} is the statement that there exists a sequence 
$\vec a = \langle a_i : i < \omega_2 \rangle$ of countable subsets of $\omega_2$ and a club 
$C \subseteq \omega_2$ such that for all limit ordinals $\alpha \in C$, 
$\alpha$ is \emph{approachable} by $\vec a$ in the following sense: 
there exists a cofinal set $c \subseteq \alpha$ with order type equal to $\cf(\alpha)$ such that 
for all $\beta < \alpha$, $c \cap \beta$ is a member of $\{ a_i : i < \alpha \}$. 
Essentially, this property is a very weak form of the square principle $\Box_{\omega_1}$. 
The failure of the $\omega_1$-approachability property is known to hold in Mitchell's 
model \cite{mitchell} in which there does not exist a special $\omega_2$-Aronszajn tree, 
which he constructed using a Mahlo cardinal.

A solution to the problem of \cite{eight} addressed in this article 
was originally discovered by the first author, using a 
mixed support forcing iteration similar to the forcings appearing in \cite{eight} and \cite{jk32}. 
Later, the second author found a different proof using the idea of a disjoint stationary sequence. 
The latter proof is somewhat easier, since it avoids the technicalities of mixed support iterations, 
and also can be easily adapted to arbitrarily large continuum. 
In this article we present the second proof.

In Section 1, we discuss the idea of a disjoint stationary sequence, which was originally introduced by the second 
author in \cite{jk11}. 
In Section 2, we prove the main result of the paper. 
In Section 3, we adapt our model to arbitrarily large continuum using an argument of I.\ Neeman, 
which we include with his kind permission.

\section{Disjoint Stationary Sequences}

Recall that for an uncountable ordinal $\alpha \in \omega_2$, $P_{\omega_1}(\alpha)$ denotes the set of 
all countable subsets of $\alpha$. 
A set $c \subseteq P_{\omega_1}(\alpha)$ is \emph{club} if it is cofinal in $P_{\omega_1}(\alpha)$ 
and closed under unions of countable increasing sequences. 
A set $s \subseteq P_{\omega_1}(\alpha)$ is \emph{stationary} if it has non-empty intersection with 
every club in $P_{\omega_1}(\alpha)$. 
For an infinite cardinal $\kappa$, 
a forcing $\p$ is said to be \emph{$\kappa$-distributive} 
if it adds no new subsets of $V$ of size less than 
$\kappa$.

Let $\alpha$ be an uncountable ordinal in $\omega_2$. 
Fix an increasing and continuous sequence $\langle b_{i} : i < \omega_1 \rangle$ 
of countable sets with union equal to $\alpha$ (for example, fix a bijection $f : \omega_1 \to \alpha$ 
and let $b_{i} := f[i]$). 
Note that the set $\{ b_i : i < \omega_1 \}$ is club in $P_{\omega_1}(\alpha)$. 
A set $s \subseteq P_{\omega_1}(\alpha)$ 
is stationary in $P_{\omega_1}(\alpha)$ 
iff the set 
$x := \{ i < \omega_1 : b_{i} \in s \}$ 
is a stationary subset of $\omega_1$. 
Indeed, if $C \subseteq \omega_1$ is a club which is disjoint from $x$, then the set 
$\{ b_i : i \in C \}$ is a club subset of $P_{\omega_1}(\alpha)$ which is obviously 
disjoint from $s$. 
On the other hand, if $c \subseteq P_{\omega_1}(\alpha)$ is a club which is disjoint from $s$,
then the set 
$\{ i < \omega_1 : b_i \in c \}$ 
is a club in $\omega_1$, and this club is clearly disjoint from $x$.

\begin{definition}
A \emph{disjoint stationary sequence} on $\omega_2$ is a sequence 
$\langle s_\alpha : \alpha \in S \rangle$, where $S$ is a stationary 
subset of $\omega_2 \cap \cof(\omega_1)$, satisfying:
\begin{enumerate}
	\item for all $\alpha \in S$, $s_\alpha$ is a stationary subset of $P_{\omega_1}(\alpha)$;
	\item for all $\alpha < \beta$ in $S$, $s_\alpha \cap s_\beta = \emptyset$.
	\end{enumerate}
\end{definition}

As we will show below, the existence of a disjoint stationary sequence 
$\langle s_\alpha : \alpha \in S \rangle$ on $\omega_2$ implies the failure of the 
$\omega_1$-approachability property (more specifically, that the set $S$ is not in 
the approachability ideal $I[\omega_2]$). 
In our main result, the failure of the $\omega_1$-approachability property will follow from the 
existence of a disjoint stationary sequence.

One of the advantages of disjoint stationary sequences over other methods for obtaining the failure 
of approachability, such as using the $\omega_1$-approximation property, is their upward absoluteness.

\begin{lemma}
	Suppose that $\langle s_\alpha : \alpha \in S \rangle$ is a disjoint stationary sequence. 
	Let $\p$ be a forcing poset which preserves $\omega_1$ and $\omega_2$, preserves the stationarity of $S$, 
	and preserves stationary subsets of $\omega_1$. 
	Then $\p$ forces that $\langle s_\alpha : \alpha \in S \rangle$ is a disjoint stationary sequence.
	\end{lemma}

The proof is straightforward.

\begin{corollary}
	Assume that $\langle s_\alpha : \alpha \in S \rangle$ is a disjoint stationary sequence. 
	Let $\p$ be a forcing poset which is either c.c.c., or $\omega_2$-distributive and preserves the 
	stationarity of $S$. 
	Then $\p$ forces that $\langle s_\alpha : \alpha \in S \rangle$ is a disjoint stationary sequence.
\end{corollary}

The next result describes a well-known consequence of approachability; we include a proof for completeness.

\begin{proposition}
	Assume that the $\omega_1$-approachability property holds. 
	Then for any stationary set $S \subseteq \omega_2 \cap \cof(\omega_1)$, there exists 
	an $\omega_2$-distributive forcing which adds a club subset of $S \cup (\omega_2 \cap \cof(\omega))$.
	\end{proposition}

\begin{proof}
	Fix a sequence $\vec a = \langle a_i : i < \omega_2 \rangle$ of countable subsets of $\omega_2$ 
	and a club $C \subseteq \omega_2$ such that for all limit ordinals $\alpha \in C$, 
	there exists a set $e \subseteq \alpha$ which is cofinal in $\alpha$, has order type $\cf(\alpha)$, 
	and for all $\beta < \alpha$, $e \cap \beta \in \{ a_i : i < \alpha \}$.
	
	Define $\p$ as the forcing poset consisting of all closed and bounded subsets of 
	$S \cup (\omega_2 \cap \cof(\omega))$, ordered by end-extension. 
	We will show that $\p$ is $\omega_2$-distributive. 
	Observe that if $c \in \p$ and $\gamma < \omega_2$, then there is $d \le c$ with 
	$\sup(d) \ge \gamma$ (for example, $d := c \cup \min(S \setminus \max \{ \sup(c), \gamma \})$). 
	Using this, a straightforward argument shows that, if $\p$ is $\omega_2$-distributive, then 
	$\p$ adds a club subset of $S \cup (\omega_2 \cap \cof(\omega))$.
	
	To show that $\p$ is $\omega_2$-distributive, fix $c \in \p$ and a family 
	$\{ D_i : i < \omega_1 \}$ of dense open subsets of $\p$. 
	We will find $d \le c$ in $\bigcap \{ D_i : i < \omega_1 \}$.
	
	Fix a regular cardinal $\theta$ large enough so that all of the objects mentioned so 
	far are members of $H(\theta)$. 
	Fix a well-ordering $\unlhd$ of $H(\theta)$. 
	Since $S$ is stationary, we can find an elementary substructure $N$ of 
	$(H(\theta),\in,\unlhd)$ such that $\vec a$, $C$, $S$, $\p$, $c$, and $\langle D_i : i < \omega_1 \rangle$ 
	are members of $N$ and $\alpha := N \cap \omega_2 \in S$. 
	In particular, $\alpha \in C \cap \cof(\omega_1)$. 
	Fix a cofinal set $e \subseteq \alpha$ with order type $\omega_1$ 
	such that for all $\beta < \alpha$, $e \cap \beta \in \{ a_i : i < \alpha \}$. 
	Enumerate $e$ in increasing order as $\langle \gamma_i : i < \omega_1 \rangle$. 
	Note that since $\{ a_i : i < \alpha \}$ is a subset of $N$ by elementarity, 
	for all $\beta < \alpha$, $e \cap \beta \in N$. 
	Consequently, for each $\delta < \omega_1$, 
	the sequence $\langle \gamma_i : i < \delta \rangle$ is a member of $N$.
	
	We define by induction a strictly descending 
	sequence of conditions $\langle c_i : i < \omega_1 \rangle$, starting 
	with $c_0 := c$, together with some auxiliary objects. 
	We will maintain that for each $\delta < \omega_1$, the sequence $\langle c_i : i < \delta \rangle$ 
	is definable in $H(\theta)$ from parameters in $N$, and hence is a member of $N$.

	Given a limit ordinal $\delta < \omega_1$, assuming that $c_i$ is defined for all $i < \delta$, 
	we define $c_{\delta,0}$ to be equal to $\bigcup \{ c_i : i < \delta \}$. 
	Then clearly $\sup(c_{\delta,0})$ is an ordinal of cofinality $\omega$. 
	Hence, $c_{\delta} := c_{\delta,0} \cup \{ \sup(c_{\delta,0}) \}$ is a condition and is a strict 
	end-extension of $c_i$ for all $i < \delta$.
	Now assume that $\xi < \omega_1$ and $c_i$ is defined for all $i \le \xi$. 
	Let $c_{\xi,0}$ be the $\unlhd$-least strict end-extension of $c_{\xi}$ such that 
	$\max(c_{\xi,0}) \ge \gamma_\xi$. 
	Now let $c_{\xi+1}$ be the $\unlhd$-least condition in $D_\xi$ which is below $c_{\xi,0}$. 
	This completes the construction.
	Define $d_0 := \bigcup \{ c_i : i < \omega_1 \}$. 
	
	Reviewing the inductive definition of the sequence $\langle c_i : i < \omega_1 \rangle$, 
	we see that for all $\delta < \omega_1$, $\langle c_i : i < \delta \rangle$ 
	is definable in $H(\theta)$ from parameters in $N$, including specifically the sequence 
	$\langle \gamma_i : i < \delta \rangle$. 
	Therefore, each $c_i$ is in $N$. 
	In addition, for each $i < \omega_1$, $\max(c_{i+1}) \ge \gamma_i$. 
	Since $\{ \gamma_i : i < \omega_1 \} = e$ is cofinal in $\alpha$, 
	$\sup(d_0) = \alpha$. 
	Let $d := d_0 \cup \{ \alpha \}$. 
	Then $d$ is a condition since $\alpha \in S$, 
	and $d \le c_i$ for all $i < \omega_1$, and in particular, $d \le c$.  
	For each $i < \omega_1$, $c_{i+1} \in D_i$, so $d \in D_i$.
	\end{proof}

\begin{proposition}
	Suppose that $\langle s_\alpha : \alpha \in S \rangle$ is a disjoint stationary sequence. 
	Then $(\omega_2 \cap \cof(\omega_1)) \setminus S$ is stationary.
	\end{proposition}

\begin{proof}
	Let $C$ be club in $\omega_2$. 
	By induction, it is easy to define an increasing and continuous 
	sequence $\langle N_i : i < \omega_1 \rangle$ satisfying:
	\begin{enumerate}
		\item each $N_i$ is a countable elementary substructure of $H(\omega_3)$ containing 
		the objects $\langle s_\alpha : \alpha \in S \rangle$ and $C$;
		\item for each $i < \omega_1$, $N_i \in N_{i+1}$.
		\end{enumerate}
	Let $N := \bigcup \{ N_i : i < \omega_1 \}$. 
	Then by elementarity, $\omega_1 \subseteq N$ and $\beta := N \cap \omega_2$ has cofinality 
	$\omega_1$ and is in $C$.
		
	We claim that $\beta \notin S$, which completes the proof. 
	Suppose for a contradiction that $\beta \in S$. 
	Then $s_\beta$ is defined and is a stationary subset of $P_{\omega_1}(\beta)$. 
	On the other hand, $\langle N_i \cap \omega_2 : i < \omega_1 \rangle$ is a club 
	subset of $P_{\omega_1}(\beta)$. 
	So we can fix $i < \omega_1$ such that $N_i \cap \omega_2 \in s_\beta$.

	Now the sequence $\langle s_\alpha : \alpha \in S \rangle$ is a member of $N$, and also 
	$N_i \cap \omega_2 \in N \cap s_\beta$. 
	So by elementarity, there exists $\alpha \in N \cap S$ such that $N_i \cap \omega_2 \in s_\alpha$. 
	Then $\alpha \in N \cap \omega_2 = \beta$, so $\alpha < \beta$. 
	Thus, we have that $N_i \cap \omega_2$ is a member of both $s_\alpha$ and $s_\beta$, 
	which contradicts that $s_\alpha \cap s_\beta = \emptyset$.
	\end{proof}

	\begin{corollary}
		Assume that there exists a disjoint stationary sequence on $\omega_2$. 
		Then the $\omega_1$-approachability property fails.
		\end{corollary}

	\begin{proof}
		Suppose for a contradiction that $\langle s_\alpha : \alpha \in S \rangle$ 
		is a disjoint stationary sequence and the $\omega_1$-approachability property holds. 
		By Proposition 1.4, fix an $\omega_2$-distributive forcing $\p$ which adds a club subset of 
		$S \cup (\omega_2 \cap \cof(\omega))$. 
		In particular, $\p$ forces that $(\omega_2 \cap \cof(\omega_1)) \setminus S$ is 
		non-stationary in $\omega_2$.
		By Proposition 1.5, the sequence $\langle s_\alpha : \alpha \in S \rangle$ is not a disjoint 
		stationary sequence in $V^\p$.
		
		Now $\p$ is $\omega_2$-distributive, and it preserves the stationarity of $S$ because 
		it adds a club subset of $S \cup (\omega_2 \cap \cof(\omega))$. 
		By Corollary 1.3, $\langle s_\alpha : \alpha \in S \rangle$ is a disjoint 
		stationary sequence in $V^\p$, which is a contradiction.
		\end{proof}
	
\section{The main result}

Assume for the rest of the section that $\kappa$ is a Mahlo cardinal. 
Without loss of generality, we may also assume that $2^\kappa = \kappa^+$, since this can 
be forced while preserving Mahloness. 
Define $S$ as the set of inaccessible cardinals below $\kappa$.

We will define a two-step forcing iteration $\p * \dot \A$ with the following properties. 
The forcing $\p$ collapses $\kappa$ to become $\omega_2$ and adds a disjoint stationary sequence on $S$. 
In $V^\p$, $\A$ is an iteration for destroying the stationarity of non-reflecting 
subsets of $\kappa \cap \cof(\omega)$. 
The forcing $\A$ will be $\kappa$-distributive and preserve the stationarity of $S$, which implies 
by Corollary 1.3 that 
there exists a disjoint stationary sequence in $V^{\p * \dot \A}$. 
Thus, in $V^{\p * \dot \A}$ we have that stationary reflection holds at $\omega_2$ and the 
$\omega_1$-approachability property fails. 
If, in addition, we assume that the Mahlo cardinal $\kappa$ is not weakly compact in $L$, then 
there exists an $\omega_2$-Aronszajn tree in $V^{\p * \dot \A}$ as discussed above.

The remainder of this section is divided into two parts. 
In the first part we will develop the forcing $\p$, and in the second we will 
handle the forcing $\A$ in $V^\p$. 
We will use the following theorem of Gitik \cite{gitik}. 
Suppose that $V \subseteq W$ are transitive 
models of \textsf{ZFC} with the same ordinals and the same 
$\omega_1$ and $\omega_2$. 
If $(P(\omega) \cap W) \setminus V$ is non-empty, then in $W$ the set 
$P_{\omega_1}(\omega_2) \setminus V$ is stationary in $P_{\omega_1}(\omega_2)$. 
For a regular cardinal $\kappa$, 
we let $\add(\kappa)$ denote the 
usual Cohen forcing consisting of 
all functions from some $\gamma < \kappa$ 
into $2$, ordered by reverse inclusion. 

We define by induction a forcing iteration 
$$
\langle \p_\alpha, \dot \q_\beta : \alpha \le \kappa, \beta < \kappa \rangle.
$$
This iteration will be a countable support forcing iteration of proper forcings. 
We will then let $\p := \p_\kappa$.

Fix $\alpha < \kappa$ and assume that $\p_\alpha$ has been defined. 
We split the definition of $\dot \q_\alpha$ into three cases. 
If $\alpha$ is an inaccessible cardinal, then let $\dot \q_\alpha$ be a $\p_\alpha$-name 
for the forcing $\add(\alpha)$. 
If $\alpha = \beta + 1$ where $\beta$ is inaccessible, then let $\dot \q_\alpha$ be a 
$\p_\alpha$-name for $\add(\omega)$. 
For all other cases, let $\dot \q_\alpha$ be a $\p_\alpha$-name for $\col(\omega_1,\omega_2)$.
Note that in any case, $\dot \q_\alpha$ is forced to be proper. 
Now let $\p_{\alpha+1}$ be $\p_\alpha * \dot \q_\alpha$.
At limit stages $\delta \le \kappa$, assuming that $\p_\alpha$ is defined for all $\alpha < \delta$, 
we let $\p_\delta$ denote the countable support limit of these forcings.

This completes the construction. 
For each $\alpha \le \kappa$, $\p_\alpha$ is a countable support iteration of proper forcings, 
and hence is proper. 
Also, by standard facts, if $\beta < \alpha$, then $\p_\beta$ is a regular suborder of $\p_\alpha$, 
and in $V^{\p_\beta}$, the quotient forcing $\p_\alpha / \dot G_{\p_\beta}$ is forcing equivalent to a 
countable support iteration of proper forcings, and hence is itself proper. 
We let $\dot{\p}_{\beta,\alpha}$ be a $\p_\beta$-name for this proper forcing iteration which is equivalent 
to $\p_\alpha / \dot G_{\p_\beta}$ in $V^{\p_\beta}$. 

One can show by well-known arguments 
that for all inaccessible cardinals $\alpha \le \kappa$, $\p_\alpha$ has size $\alpha$, 
is $\alpha$-c.c., and forces that $\alpha = \omega_2$. 
Namely, since $\alpha$ is inaccessible, for all $\beta < \alpha$, $|\p_\beta|< \alpha$. 
Hence $\p_\alpha$ has size $\alpha$ by definition. 
A standard $\Delta$-system argument shows that $\p_\alpha$ is $\alpha$-c.c., and since 
collapses are used at cofinally many stages below $\alpha$, $\p_\alpha$ turns $\alpha$ into $\omega_2$.

Let $\p := \p_{\kappa}$. 
In $V^{\p}$, let us define a disjoint stationary sequence. 
Recall that $S$ is the set of inaccessible cardinals in $\kappa$ in the ground model $V$. 
Since $\kappa$ is Mahlo, $S$ is a stationary subset of $\kappa$ in $V$.
As $\p$ is $\kappa$-c.c., $S$ remains stationary in $V^{\p}$. 
And since $\p$ is proper and forces that $\kappa = \omega_2$, 
each member of $S$ has cofinality $\omega_1$ in $V^{\p}$.

The set $S$ will be the domain of the disjoint stationary sequence in $V^{\p}$. 
Consider $\alpha \in S$. 
Then $\p_\alpha$ forces that $\alpha = \omega_2$. 
We have that $\p_{\alpha+1}$ is forcing equivalent to $\p_\alpha * \add(\alpha)$ and 
$\p_{\alpha+2}$ is forcing equivalent to 
$$
\p_\alpha * \add(\alpha) * \add(\omega).
$$
Clearly, $\alpha$ is still equal to $\omega_2$ after forcing with $\p_{\alpha+1}$ or 
$\p_{\alpha+2}$.

Since there exists a subset of $\omega$ in 
$V^{\p_{\alpha+2}} \setminus V^{\p_{\alpha+1}}$, 
in $V^{\p_{\alpha+2}}$ the set 
	$$
	s_\alpha := P_{\omega_1}(\alpha) \setminus V^{\p_{\alpha+1}}
	$$
is a stationary subset of $P_{\omega_1}(\alpha)$ by Gitik's theorem. 
Now the tail of the iteration $\p_{\alpha+2,\kappa}$ is proper in $V^{\p_{\alpha+2}}$. 
Therefore, $s_\alpha$ remains stationary in $P_{\omega_1}(\alpha)$ in $V^{\p}$.

Observe that if $\alpha < \beta$ are both in $S$, then by definition 
$s_\alpha \subseteq V^{\p_{\alpha+2}} \subseteq V^{\p_\beta}$, whereas 
$s_\beta \cap V^{\p_\beta} = \emptyset$. 
Thus, $s_\alpha \cap s_\beta = \emptyset$. 
It follows that in $V^{\p}$, 
$\langle s_\alpha : \alpha \in S \rangle$ 
is a disjoint stationary sequence on $\omega_2$.

\bigskip

For the second part of our proof, we work in $V^\p$ to define a forcing iteration $\A$ of length $\kappa^+$ 
which is designed to destroy the stationarity of any subset of 
$\omega_2 \cap \cof(\omega)$ which does not reflect to an ordinal in $\omega_2 \cap \cof(\omega_1)$. 
This forcing will be shown to be $\kappa$-distributive and preserve the stationarity of $S$. 
It follows from Corollary 1.3 that $\A$ preserves the fact that $\langle s_\alpha : \alpha \in S \rangle$ 
is a disjoint stationary sequence. 
Note that since $\p$ is $\kappa$-c.c.\ and has size $\kappa$, easily $2^\kappa = \kappa^+$ in $V^\p$.

The definition of and arguments involving $\A$ are essentially the same as in the original construction 
of Harrington and Shelah \cite{HS}. 
The main differences are that we are using $\p$ to collapse $\kappa$ to become $\omega_2$ 
instead of $\col(\omega_1,<\! \kappa)$, and that we are now required to show that $\A$ 
preserves the stationarity of $S$. 
We will sketch the main points of the construction, 
but leave some of the routine technical details to be 
checked by the reader in consultation with \cite{HS}.

\bigskip

Many of the facts which we will need to know 
about $\A$ can be abstracted out more generally 
to a kind of forcing iteration which we will call a 
suitable iteration. 
So before defining $\A$, let us describe 
this kind of iteration in detail.
We will assume in what follows that 
$2^{\omega_1} = \omega_2$.

Let us define abstractly the idea of a \emph{suitable iteration} 
$$
\langle \A_i, \dot T_j : i \le \alpha, j < \alpha \rangle,
$$
where $\alpha \le \omega_3$. 
Such an iteration is determined by the following 
recursion. 
A condition in $\A_i$ is any function $p$ 
whose domain is a subset of $i$ of size less than $\omega_2$ 
such that for all $j \in \dom(p)$, $p(j)$ is a 
non-empty closed and bounded subset of $\omega_2$ such that 
$p \restrict j$ forces in $\A_j$ that 
$p(j) \cap \dot T_j = \emptyset$. 
We let $q \le p$ if $\dom(p) \subseteq \dom(q)$ 
and for all $i \in \dom(p)$, $q(i)$ is an end-extension 
of $p(i)$. 
And $\dot T_i$ is a nice $\A_i$-name for a subset of 
$\omega_2 \cap \cof(\omega)$.\footnote{In our construction 
below, our specific suitable iteration will be shown to be 
$\omega_2$-distributive. 
However, being $\omega_2$-distributive is not a part of 
the abstract definition of a suitable iteration.}

Suppose that $M$ is a transitive model of  
$\textsf{ZFC}^-$ which is 
closed under $\omega_1$-sequences. 
Then if $M$ models that 
$\langle \A_i, \dot T_j : i \le \alpha, j < \alpha \rangle$ 
is a suitable iteration, then in fact it is. 
Specifically, all the notions used in the recursion above are 
upwards absolute for such a model, since $M$ contains 
all $\omega_1$-sized sets. 
For example, $M$ contains all closed and bounded subsets of 
$\omega_2$ and being a nice name is absolute. 

Observe that if $\alpha < \omega_3$, then 
$2^{\omega_1} = \omega_2$ immediately 
implies that $\A_\alpha$ has size $\omega_2$. 
On the other hand, if $\alpha = \omega_3$, then a straightforward 
application of the $\Delta$-system lemma shows that $\A_{\omega_3}$ is $\omega_3$-c.c. 
Using a covering and nice name argument, 
it then follows that if 
$\A_\beta$ is $\omega_2$-distributive for all $\beta < \omega_3$, then so is $\A_{\omega_3}$. 

\begin{lemma}
	Suppose that 
	for all $i < \alpha$, 
	$\A_i$ forces that $\dot T_i$ is non-stationary. 
	Then for any $q \in \A_\alpha$, 
	$\A_\alpha / q$ is forcing equivalent 
	to $\add(\omega_2)$.
	\end{lemma}

\begin{proof}
	First we claim that $\A_\alpha$ contains an 
	$\omega_2$-closed dense subset. 
	For each $i$ let $\dot E_i$ be an $\A_i$-name for a club disjoint from $\dot T_i$. 
	Define $D$ as the set of conditions $p$ 
	such that for all $i \in \dom(p)$, 
	$p \restrict i$ forces that $\max(p(i)) \in \dot E_i$. 
	It is easy to 
	prove that $D$ is dense and $\omega_2$-closed.
	
	Reviewing the definition of $\A_\alpha$, 
	clearly $\A_\alpha$ is separative 
	and every condition in it has $\omega_2$-many 
	incompatible extensions. 
	By a well-known fact, any $\omega_2$-closed separative 
	forcing of size $\omega_2$ 
	for which any condition 
	has $\omega_2$-many 
	incompatible extensions is forcing equivalent to 
	$\add(\omega_2)$.
\end{proof}

Having described the main facts which 
we will use about 
a suitable iteration, let us show how this kind of 
iteration can be used to obtain a model satisfying 
that stationary reflection holds at $\omega_2$. 
Suppose that we have a ground model in which 
$2^{\omega_2} = \omega_3$. 
Using a standard bookkeeping argument, we can define a 
suitable iteration 
$$
\langle \A_i, \dot T_j : i \le \omega_3, 
j < \omega_3 \rangle,
$$
so that every nice name for a non-reflecting subset of 
$\omega_2 \cap \cof(\omega)$ is equal to some $\dot T_j$. 
Specifically, assuming that 
$\A_i$ is defined for some $i < \omega_3$, 
then using $2^{\omega_2} = \omega_3$ and the fact that 
$\A_i$ has size $\omega_2$, 
we can list out all nice $\A_i$-names for subsets of 
$\omega_2 \cap \cof(\omega)$ in order type $\omega_3$. 
Now choose $\dot T_i$ to be the first name (according to 
the bookkeeping function) 
which was listed at some stage less than or equal to $i$ 
which is forced by $\A_i$ to be non-reflecting. 
In this manner, we can 
arrange that after $\omega_3$-many stages, 
all names which arise during the iteration are handled, 
and thus that the iteration destroys the stationarity of all non-reflecting sets. 
Of course this construction breaks down if we reach some 
$i$ such that $\A_i$ is not $\omega_2$-distributive. 
So proving the $\omega_2$-distributivity of such a suitable 
iteration will be the main remaining goal.

\bigskip

This completes the abstract description of a suitable iteration and how it will be 
used to obtain stationary reflection at $\omega_2$. 
Let us now return to our construction. 
Fix a generic filter $G$ on $\p$. 
Then in $V[G]$ we have that 
$\kappa = \omega_2$, $2^{\omega_1} = \omega_2$, and $2^{\omega_2} = \omega_3 = \kappa^+$. 
Working in $V[G]$, we define a suitable iteration 
$\langle \A_i, \dot T_j : i \le \kappa^+, j < \kappa^+ \rangle$. 
We will prove that each $\A_i$ is $\omega_2$-distributive and preserves the stationarity of $S$. 
By the discussion above, this will complete the proof of our main result.

Fix $\alpha < \kappa^+$. 
In $V$, fix $\p$-names $\dot \A_i$ for all $i \le \alpha$ and $\dot T_j$ for all $j < \alpha$ 
which are forced to satisfy the definitions of these 
objects given above (we will abuse notation 
by writing $\dot T_j$ for the $\p$-name for the $\A_j$-name $\dot T_j$).

We would like to prove that $\A_\alpha$ is $\kappa$-distributive and 
preserves the stationarity of $S$. 
In order to prove this, we will make two inductive hypotheses. 
The first inductive hypothesis is that for all $\beta < \alpha$, $\A_\beta$ is $\kappa$-distributive 
and preserves the stationarity of $S$. 

Before describing the second inductive hypothesis, we need to develop some ideas and notation. 
For each $\beta \le \alpha$, define in $V$ the set $\mathcal X_\beta$ 
to consist of all sets $N$ satisfying:
\begin{enumerate}
	\item $N \prec H(\kappa^+)$;
	\item $N$ contains as members $\p$ and $\langle \dot \A_i, \dot T_j : i \le \beta, j < \beta \rangle$;
	\item $\kappa_N := |N| = N \cap \kappa$ and $N^{< \kappa_N} \subseteq N$;
	\item $\kappa_N \in S$.
	\end{enumerate}
An easy application of the stationarity of $S$ and the inaccessibility of $\kappa$ shows that 
each $\mathcal X_\beta$ is a stationary subset of $P_{\kappa}(H(\kappa^+))$. 
Also note that if $N \in \mathcal X_\beta$ and $\gamma \in N \cap \beta$, then 
$N \in \mathcal X_\gamma$.

Consider $N$ in $\mathcal X_\alpha$. 
Since $\p$ is $\kappa$-c.c., the maximal condition in $\p$ is $(N,\p)$-generic. 
So if $G$ is a $V$-generic filter on $\p$, then $N[G] \cap V = N$. 
In particular, $N[G] \cap \kappa = N \cap \kappa = \kappa_N \in S$.
Let $\pi : N[G] \to \overline{N[G]}$ be the transitive collapsing map 
of $N[G]$ in $V[G]$. 
Let $G^* := G \cap \p_{\kappa_N}$, which is a $V$-generic filter on $\p_{\kappa_N}$. 

\begin{lemma}
	The following statements hold.
	\begin{enumerate}
		\item $\pi \restrict N : N \to \overline{N}$ is the transitive collapsing map of $N$ in $V$;
		\item $\pi(\p) = \p_{\kappa_N}$, $\pi(G) = G^*$, and $\overline{N[G]} = \overline{N}[G^*]$; 
		in particular, $\overline{N[G]}$ is a member of $V[G^*]$;  
		\item $\overline{N[G]} = \overline{N}[G^*]$ is closed under $< \kappa_N$-sequences in $V[G^*]$.
		\end{enumerate}
	\end{lemma}

\begin{proof}
	(1) and (2) are straightforward. 
	Since $\overline{N}^{< \kappa_N} \subseteq \overline{N}$ in $V$ by the closure of $N$ 
	and $\p_{\kappa_N}$ is $\kappa_N$-c.c., (3) follows immediately by a standard fact.
	\end{proof}

Now we are ready to state our second inductive hypothesis: 
for all $\beta < \alpha$ and for all $N \in \mathcal X_\beta$, letting 
$\pi : N[G] \to \overline{N[G]}$ be the transitive collapsing map of $N[G]$ and $G^* := \pi(G)$, 
for all $q \in \pi(\A_\beta)$, the forcing poset 
$\pi(\A_\beta) / q$ is forcing equivalent to $\add(\omega_2)$ in $V[G^*]$.

\bigskip

We begin the proof of the two inductive hypotheses for $\alpha$, assuming that they hold 
for all $\beta < \alpha$. 
Let $N \in \mathcal X_\alpha$. 
Let $\pi : N[G] \to \overline{N[G]}$ 
be the transitive collapsing map of 
$N[G]$ and $G^* := \pi(G)$. 
Since $\pi$ is an isomorphism, by the absoluteness of suitable iterations we have that in $V[G^*]$, 
$$
\langle \A_i^*, \dot T_j^* : i \le \pi(\alpha), j < \pi(\alpha) \rangle := 
\pi(\langle \A_i, \dot T_j : i \le \alpha, j < \alpha \rangle)
$$
is a suitable iteration of length $\pi(\alpha) < \omega_3$. 
Applying Lemma 2.1 to this suitable iteration in 
the model $V[G^*]$, 
the second inductive hypothesis for $\alpha$ will follow 
from the next lemma.

\begin{lemma}
	For all $\gamma \in N \cap \alpha$, 
	$\pi(\A_\gamma) = \A_{\pi(\gamma)}^*$ 
	forces over $V[G^*]$ that 
	$\pi(\dot T_\gamma) = \dot T_{\pi(\gamma)}^*$ 
	is non-stationary in $\kappa_N$.
	\end{lemma}

\begin{proof}
Consider $\gamma \in N \cap \alpha$. 
Then by the choice of the names used in the iteration, 
$\A_\gamma$ forces that $\dot T_\gamma$ is a subset of $\kappa \cap \cof(\omega)$ 
which does not reflect to any ordinal in $\kappa \cap \cof(\omega_1)$. 
In particular, $\A_\gamma$ forces that $\dot T_\gamma \cap \kappa_N$ is non-stationary in $\kappa_N$.

Consider $q \in \pi(\A_\gamma)$. 
We will find a $V[G^*]$-generic filter $H$ on $\pi(\A_\gamma)$ which contains $q$ 
such that in $V[G^*][H]$, $\pi(\dot T_\gamma)^H$ is non-stationary in $\kappa_N$. 
Because $q$ is arbitrary, 
this proves that $\pi(\A_\gamma)$ forces that $\pi(\dot T_\gamma)$ is non-stationary. 
Since $N$ is in $\mathcal X_\alpha$ and $\gamma \in N \cap \alpha$, 
$N$ is in $\mathcal X_\gamma$. 
By the second inductive hypothesis, $\pi(\A_\gamma) / q$ is forcing equivalent to $\add(\kappa_N)$ in $V[G^*]$. 
By definition, the forcing iteration $\p$ forces with 
$\add(\kappa_N)$ at stage $\kappa_N$. 
Hence, we can write $V[G \cap \p_{\kappa_N + 1}]$ as 
$V[G^*][H]$, where $H$ is some $V[G^*]$-generic filter on $\pi(\A_\gamma) / q$.

Now $\pi \restrict \A_\gamma$ is an isomorphism between the posets $N[G] \cap \A_\gamma$ and $\pi(\A_\gamma)$. 
Therefore, $I := \pi^{-1}(H)$ is a filter on $N[G] \cap \A_\gamma$. 
The fact that $H$ is a $V[G^*]$-generic filter on $\pi(\A_\gamma)$ easily implies that 
$I$ meets every dense subset of $\A_\gamma$ which 
is a member of $N[G]$. 
Now a lower bound $t$ of $I$ can be easily constructed 
by taking the coordinate-wise closure of the union of 
the clubs appearing in the conditions of $I$. 
Namely, the 
fact that $I$ meets every dense set in $N[G]$ 
implies that the 
maximum member of any such club is equal to $\kappa_N$, 
which has cofinality $\omega_1$ in $V[G]$ and hence 
is not in any of the sets $\dot T_j$.

Fix a $V[G]$-generic filter $h$ on $\A_\gamma$ 
which contains $t$. 
Now $\pi^{-1} : \overline{N}[G^*] \to N[G]$ is an elementary embedding of $\overline{N}[G^*]$ 
into $H(\kappa^+)^{V[G]}$ which satisfies that $\pi^{-1}(H) = I \subseteq h$. 
So by a standard fact about extending elementary 
embeddings, 
we can extend $\pi^{-1}$ to an elementary embedding 
$\tau : \overline{N}[G^*][H] \to N[G][h]$ 
which maps $H$ to $h$. 
Let $T^* := \pi(\dot T_\gamma)^H$ and $T_\gamma := (\dot T_\gamma)^h$. 
Then clearly, $\tau(T^*) = T_\gamma$.

Since $\kappa_N$ is the critical point of $\tau$, 
$T_\gamma \cap \kappa_N = T^*$. 
As $\A_\gamma$ forces that $\dot T_\gamma$ does not reflect to $\kappa_N$, 
$T^*$ is a non-stationary subset of $\kappa_N$ in the model $V[G][h]$. 
By the first inductive hypothesis, $\A_\gamma$ is $\kappa$-distributive. 
Therefore, any club of $\kappa_N$ 
in $V[G][h]$ is actually in $V[G]$. 
Thus, $T^*$ is non-stationary in $V[G]$. 
But $V[G]$ is a generic extension of $V[G^*][H]$ by the proper forcing $\p_{\kappa_N + 1,\kappa}$. 
So $T^*$ is non-stationary in $V[G^*][H]$.
\end{proof}

This completes the proof of the second inductive 
hypothesis.
It remains to prove the first inductive hypothesis that 
$\A_\alpha$ is $\kappa$-distributive and preserves 
the stationarity of $S$. 

\begin{lemma}
	For all $N \in \mathcal X_\alpha$, 
	for all $a \in N[G] \cap \A_\alpha$, there exists a filter $I$ on $N[G] \cap \A_\alpha$ 
	in $V[G]$ containing $a$ 
	which meets every dense subset of 
	$\A_\alpha$ in $N[G]$. 
	\end{lemma}

\begin{proof}
	This is similar to a part of the proof of the Lemma 2.3. 
	Let $\pi : N[G] \to \overline{N[G]}$ be the transitive collapsing map of $N[G]$ and $G^* := \pi(G)$. 
	Let $a \in N[G] \cap \A_\alpha$. 
	Then $\pi(a) \in \pi(\A_\alpha)$. 
	By the second inductive hypothesis which we have 
	now verified for $\alpha$, 
	$\pi(\A_\alpha) / \pi(a)$ 
	is forcing equivalent to $\add(\kappa_N)$ in $V[G^*]$. 
	By definition, the forcing iteration $\p$ 
	forces with $\add(\kappa_N)$ at stage $\kappa_N$. 
	Hence, we can write $V[G \cap \p_{\kappa_N + 1}]$ as 
	$V[G^*][H]$, where $H$ is some $V[G^*]$-generic filter on $\pi(\A_\alpha) / \pi(a)$.

	Now $\pi \restrict \A_\alpha$ 
	is an isomorphism between the posets 
	$N[G] \cap \A_\alpha$ and $\pi(\A_\alpha)$. 
	Therefore, $I := \pi^{-1}(H)$ is a filter on 
	$N[G] \cap \A_\alpha$. 
	The fact that $H$ is a $V[G^*]$-generic filter on 
	$\pi(\A_\alpha)$ easily implies that 
	$I$ meets every dense subset of $\A_\alpha$ which 
	is a member of $N[G]$.
	\end{proof}

We can now complete the proof that 
$\A_\alpha$ is $\kappa$-distributive 
and preserves the stationarity of $S$. 
Given a family $\mathcal D$ 
of fewer than $\kappa$ many dense open 
subsets of $\A_\alpha$ and a condition $a \in \A_\alpha$, 
we may pick $N \in \mathcal X_\alpha$ so that $\mathcal D$ 
and $a$ are members of $N[G]$. 
Then $\mathcal D \subseteq N[G]$. 
By Lemma 2.4, fix a 
filter $I$ on $N[G] \cap \A_\alpha$ 
in $V[G]$ which contains $a$ and meets every 
dense subset of $\A_\alpha$ in $N[G]$ (and in particular, 
meets every dense set in $\mathcal D$). 
It is easy to define a lower bound $t$ of $I$ in 
$\A_\alpha$ by taking the coordinate-wise 
closure of the union of the clubs appearing in 
the conditions in $I$. 
Then $t \le a$ and $t$ is in every dense open set 
in $\mathcal D$.

Similarly, given an $\A_\alpha$-name $\dot C$ for a club subset of $\kappa$ and $a \in \A_\alpha$, 
we may choose $N \in \mathcal X_\alpha$ such that $\dot C$ and $a$ are in $N$. 
Fix a filter $I$ on $N[G] \cap \A_\alpha$ 
in $V[G]$ which contains $a$ and 
meets every dense subset of 
$\A_\alpha$ in $N[G]$. 
As usual, let $t$ be a lower bound of $I$. 
Then $t$ is an $(N[G],\A_\alpha)$-generic condition, 
which implies that $t$ forces that $N[G] \cap \kappa = \kappa_N$ is in $S \cap \dot C$.

\section{Arbitrarily large continuum}

In the model of the previous section, $2^\omega = \omega_2$ holds. 
A violation of \textsf{CH} is necessary, since $\textsf{CH}$ implies the 
$\omega_1$-approximation property, as witnessed by any enumeration 
of all countable subsets of $\omega_2$ in order type $\omega_2$. 
In this section, we will show how to modify this model to obtain 
arbitrarily large continuum. 
This modification will use an unpublished result of I.\ Neeman.

\begin{thm}[Neeman]
	Assume that stationary reflection holds at $\omega_2$. 
	Then for any ordinal $\mu$, 
	$\mathrm{Add}(\omega,\mu)$ forces that stationary reflection still holds at $\omega_2$.
	\end{thm}

\begin{proof}
	We first prove the result in the special case that $\mu = \omega_2$. 
	Let $p \in \add(\omega,\omega_2)$, and suppose that $p$ forces that 
	$\dot S$ is a stationary subset of $\omega_2 \cap \cof(\omega)$. 
	We will find $q \le p$ and an ordinal $\beta \in \omega_2 \cap \cof(\omega_1)$ such that 
	$q$ forces that $\dot S \cap \beta$ is stationary in $\beta$.
	
	Let $T$ be the set of ordinals $\alpha < \omega_2$ 
	such that for some $s \le p$, $s$ forces that $\alpha \in \dot S$. 
	Then $T \subseteq \omega_2 \cap \cof(\omega)$. 
	An easy observation is that $p$ forces that $\dot S \subseteq T$, and consequently 
	$T$ is a stationary subset of $\omega_2$. 
	For each $\alpha \in T$, fix a witness $s_\alpha \le p$ which forces that $\alpha \in \dot S$, and define 
	$$
	a_\alpha := s_\alpha \restrict (\alpha \times \omega) \ \textrm{and} \ 
	b_\alpha := s_\alpha \restrict ([\alpha,\omega_2) \times \omega).
	$$
	
	Using Fodor's lemma, we can find a stationary set $U \subseteq T$ 
	and a set $x$ satisfying that for all $\alpha \in U$, $a_\alpha = x$. 
	Observe that $q := x \cup p$ is a condition which extends $p$. 
	Applying the fact that stationary reflection holds in the ground model together with an easy closure argument, 
	we can fix $\beta \in \omega_2 \cap \cof(\omega_1)$ 
	such that $U \cap \beta$ is stationary in $\beta$ and for all $\alpha < \beta$, 
	$\dom(s_\alpha) \subseteq \beta \times \omega$.

	We claim that $q$ forces that $\dot S \cap \beta$ is stationary in $\beta$, which finishes the proof. 
	Suppose for a contradiction that there is $r \le q$ which forces that $\dot S \cap \beta$ is non-stationary in $\beta$. 
	Using the fact that $\add(\omega,\omega_2)$ is c.c.c.\ and $\cf(\beta) = \omega_1$, there exists 
	a club $D \subseteq \beta$ in the ground model such that $r$ forces that $D \cap \dot S = \emptyset$. 
	As $r$ is finite, we can fix $\delta < \beta$ such that 
	$\dom(r) \cap (\beta \times \omega) \subseteq \delta \times \omega$.

	Since $U \cap \beta$ is stationary in $\beta$, fix 
	$\alpha \in U \cap D$ larger than $\delta$. 
	We claim that $s_\alpha$ and $r$ are compatible. 
	By the choice of $U$, $s_\alpha \restrict (\alpha \times \omega) = x$, and by the choice of $\beta$, 
	$\dom(s_\alpha) \subseteq \beta \times \omega$. 
	Suppose that $(\xi,n) \in \dom(s_\alpha) \cap \dom(r)$. 
	Then $\xi < \beta$, so $(\xi,n) \in \dom(r) \cap (\beta \times \omega) \subseteq \delta \times \omega$. 
	Thus, $\xi < \delta < \alpha$. 
	So $(\xi,n) \in \alpha \times \omega$, and hence $s_\alpha(\xi,n) = a_\alpha(\xi,n) = x(\xi,n)$. 
	On the other hand, $r \le q \le x$, and so 
	$r(\xi,n) = x(\xi,n) = s_\alpha(\xi,n)$.

	This proves that $r$ and $s_\alpha$ are compatible. 
	Fix $t \le r, s_\alpha$. 
	Since $t \le s_\alpha$, $t$ forces that $\alpha \in \dot S$. 
	On the other hand, $\alpha \in D$, and $r$ forces that $\dot S \cap D = \emptyset$. 
	So $r$, and hence $t$, forces that $\alpha \notin \dot S$, which is a contradiction.
	
	Now we prove the result for arbitrary ordinals $\mu$. 
	If $\mu < \omega_2$, then $\add(\omega,\omega_2)$ is isomorphic to 
	$\add(\omega,\mu) \times \add(\omega,\omega_2 \setminus \mu)$. 
	Since stationary reflection holds in $V^{\add(\omega,\omega_2)}$, it also holds in 
	the submodel $V^{\add(\omega,\mu)}$, since a non-reflecting stationary set in the latter model would 
	remain a non-reflecting stationary set in the former model.
	
	Suppose that $\mu > \omega_2$. 
	Let $p$ be a condition in $\add(\omega,\mu)$ which forces that $\dot S$ 
	is a stationary subset of $\omega_2 \cap \cof(\omega)$, for some nice name $\dot S$. 
	Then by the c.c.c.\ property of $\add(\omega,\mu)$ and the fact that conditions are finite, 
	it is easy to show there exists a set $X \subseteq \mu$ of size $\omega_2$ 
	such that $\dot S$ is a nice $\add(\omega,X)$-name and $p \in \add(\omega,X)$. 
	Since $X$ has size $\omega_2$, $\add(\omega,X)$ is isomorphic to $\add(\omega,\omega_2)$. 
	By the first result above, we can find $q \le p$ in $\add(\omega,X)$ and $\beta \in \omega_2 \cap \cof(\omega_1)$ 
	such that $q$ forces in $\add(\omega,X)$ that $\dot S \cap \beta$ is stationary in $\beta$. 
	Since $\add(\omega,\mu)$ is isomorphic to $\add(\omega,X) \times \add(\omega,\mu \setminus X)$ 
	and $\add(\omega,\mu \setminus X)$ is c.c.c.\ in $V^{\add(\omega,X)}$, an easy argument shows that 
	$q$ forces in $\add(\omega,\mu)$ that $\dot S \cap \beta$ is stationary in $\beta$.
	\end{proof}

Now start with the model $W := V^{\p * \dot \A}$ from the previous section. 
Then $\omega_2$ is not weakly compact in $L$, there exists a disjoint stationary sequence in $W$,  
and stationary reflection holds at $\omega_2$ in $W$. 
Let $\mu$ be any ordinal and let $H$ be a $W$-generic filter on $\add(\omega,\mu)$.  
Since $\add(\omega,\mu)$ is c.c.c., Corollary 1.3 implies that there exists a disjoint stationary sequence in $W[H]$. 
As $\omega_2$ is not weakly compact in $L$, there exists an $\omega_2$-Aronszajn tree in $W[H]$. 
And stationary reflection holds in $W[H]$ by Theorem 3.1.

\bibliographystyle{plain}
\bibliography{paper35}

\end{document}